\documentclass[11pt,a4paper]{article}
\usepackage[english]{babel}
\usepackage{amsmath,amssymb,amsxtra,amsthm}

\usepackage{times}

\newcommand{\Thue}{\mathbf{t}}

\newtheorem{theorem}{Theorem}

\newtheorem{lemma}[theorem]{Lemma}

\newtheorem{corollary}{Corollary}

\theoremstyle{definition}

\newcommand{\EXTRA}[1]{}

\begin{document}

\title{Square-free Walks on Labelled Graphs}

\author{Tero Harju\\
         Department of Mathematics\\
         University of~Turku, Finland\\
         \texttt{harju@utu.fi}
         }

\maketitle

\begin{abstract}
A finite or infinite word is called  a $G$-word for a labelled graph $G$ on the vertex set $A_n = \{0,1, \ldots, n-1\}$ if
$w = i_1i_2 \cdots i_k \in A_n^*$, where each factor $i_ji_{j+1}$ is an edge of $E$, i.e,
$w$ represents  a walk in $G$.
We show that there exists a square-free infinite $G$-word if and only if $G$ has no subgraph
isomorphic to one of the cycles $C_3, \ C_4, \ C_5$, the path $P_5$ or the claw $K_{1,3}$.
The colour number $\gamma(G)$ of a graph $G=(A_n,E)$ is the smallest integer~$k$, if it exists,
for which there exists a mapping $\varphi\colon A_n \to A_k$ such that
$\varphi(w)$ is square-free for an infinite $G$-word $w$.
We show that $\gamma(G)=3$ for $G=C_3, C_5, P_5$, but $\gamma(G)=4$
for $G=C_4, K_{1,3}$. In particular, $\gamma(G) \leq 4$ for all graphs that
have at least five vertices.
\end{abstract}

\bigskip

\noindent\textbf{Note}
The author is thankful for James Currie for pointing out that,
apart from some technical differences,
the results of the present paper are covered by his paper

\smallskip
\noindent
James   D.  Currie:  Which  graphs  allow  infinite
nonrepetitive  walks?
\emph{Discrete Math.}  87  (1991)   249--260.

\bigskip


\section{Introduction}

Alon et al.~\cite{Alon} initiated the study of non-repetiveness of simple paths in
edge coloured graphs.
In~\cite{Alon} an edge  colouring  of a graph $G$ is
square-free  if the sequence of colours on every simple path in $G$
consists of a square-free word.
The problem for vertex coloured graphs was consider, e.g.,
by  Bar\'at and  Varj{\'u}~\cite{BaratVarju}.
For results on these settings, see also~\cite{BaratWood,Bresar1, Bresar2,Grytczuk}.
In the present paper we consider the problem when a vertex colouring of a graph has an infinite
square-free walk.

We consider square-freeness of \emph{infinite words} $w\colon \mathbb{N} \to A_n$
over the alphabets
$A_n = \{0,1, \ldots, n-1\}$ of cardinality $n$.
An infinite word $w$ will be represented as a sequence of letters from $A_n$,
i.e., $w=i_1i_2 \cdots$ with $i_j \in A_n$ for all $j \in \mathbb{N}$.
A (finite) \emph{word} is a finite sequence of letters.
The set of all finite words over $A_n$ is denoted by $A_n^*$.
This set contains the empty word.
The length of a word $w$ is denoted by $|w|$.
If $w=u_1vu_2$, then $v$ is a \emph{factor} of $w$.
If here $u_1$ is the empty word, then $v$ is a \emph{prefix} of $w$ and if $u_2$ is empty,
then $v$ is a \emph{suffix} of $w$.
The above terminology generalizes to infinite words in a natural way.

A finite or infinite word $w$ is \emph{square-free} if it does not contain any factor
of the form $u^2 =uu$ for a nonempty word $u$.
Thue~\cite{Thue} showed a hundred years ago that there are infinite
square-free words over the ternary alphabet $A_3$; see Lothaire~\cite{Lothaire} or
the historical survey by Berstel and Perrin~\cite{BerstelPerrin}.
An example of an infinite square-free word can be obtained by iterating the morphism
\[
\tau(0)=012, \quad \tau(1)=02, \quad \tau(2)=1
\]
starting from the letter $0 \in A_3$. The result is the infinite  word
\begin{equation}\label{thue}
\Thue=012021012102012021020121012 \cdots\,.
\end{equation}
Observe that $\Thue$ does not have factors $010$ and $212$ .
We call $\Thue$ the \emph{Thue word} although  it is due to Istrail~\cite{Istrail}.

A mapping  $\alpha\colon A^* \to B^*$ between word sets over the alphabets $A$
and $B$ is a \emph{morphism}, if $\alpha(uv)=\alpha(u)\alpha(v)$ for
all words $u,v \in A^*$. Clearly, each morphism $\alpha\colon A^* \to B^*$
is determined by its images $\alpha(a)$ of the letters $a \in A$.
A~morphism $\alpha$ is \emph{square-free}, if it preserves square-freeness of words,
i.e., if $v\in A^*$ is square-free, then so is the image $\alpha(v) \in A^*$.

Let $G=(A_n,E)$ be an undirected graph on the vertex set $A_n$ such that
the edge set $E$ does not contain self-loops from a vertex $i$ to itself
nor parallel edges between the same vertices. An edge between $i$ and $j$ is
denoted by $ij$.
We say that a word $w = i_1i_2 \cdots i_k \in A_n^*$ is a $G$-\emph{word} if
each factor $i_ji_{j+1}$ is an edge of $E$. Thus a $G$-word is a walk on $G$.

A $k$-\emph{colouring} of a graph $G$ on $A_n$ is a morphism
$\varphi\colon A_n^* \to A_k^*$.
The graph $G$ has an \emph{infinite $k$-coloured square-free walk}, if
there is an infinite $G$-word $w$ such that $\varphi(w)$ is  square-free
for some $k$-colouring $\varphi$. We denote by $\gamma(G)$ the smallest~$k$, if it exists,
for which~$G$ has an infinite $k$-coloured square-free $G$-word.

Let $P_n$ denote the (simple) path on the vertices $A_n$ with the $n-1$ edges $i(i+1)$,
for $i=0,1, \ldots, n-2$.
Similarly, $C_n$ denotes  the cycle on $A_n$ with the $n$ edges $i(i+1)$,
where the indices are modulo $n-1$.
Finally, let $K_{1,3}$ be the graph on $A_4$, called the \emph{claw},
with the edges $01,02,03$.

The following theorem is our main result. It will be proven in the next section.
In particular, we show that, for connected graphs,
there exists an infinite square-free $G$-word if the graph $G$
is a cycle $C_m$, for $m \geq 3$, or a path $P_m$, for $m \geq 5$, or
it contains the claw $K_{1,3}$ as an induced subgraph.

\begin{theorem}
There exists an infinite square-free $G$-word if and only if the graph $G$
has a connected subgraph isomorphic to one of the graphs
\[
C_3, \ C_4, \ C_5, \ P_5, \ K_{1,3}\,.
\]
Moreover, if $G$ has a subgraph $C_m$ with $m =3$ or $P_n$ with $n \geq 5$,
then $\gamma(G)=3$, and, otherwise, if $G$ has $G=C_4$ or $K_{1,3}$
then $\gamma(G)=4$.
\end{theorem}

\begin{corollary}
For  each graph $G$ of at least five vertices, $\gamma(G) \leq 4$.
\end{corollary}

\begin{proof}
If $G$ does not have a subgraph $K_{1,3}$, then the maximum degree
of $G$ is at most two, and in this case the connected components of $G$
are either paths $P_m$ or cycles $C_k$.  The claim follows, since each path or cycle of five
vertices has a subgraph $P_5$, $C_3$,  $C_4$ or $C_5$.
(Note that $C_k$ for $k \geq 6$ contains $P_5$.)
\end{proof}

\section{Proof of the theorem}

Recall that in general a subgraph $H$ of a graph $G$ need not be induced, i.e.,
$H$ might miss some edges of $G$ that are between vertices of $H$.

\goodbreak
\begin{lemma}
Let $G$ be a graph. The following cases are equivalent.

\begin{itemize}
\item[(i)]
There exists an infinite square-free $G$-word.

\item[(ii)]
$G$ has a connected subgraph $H$
that has an infinite square-free $H$-word.
\end{itemize}
\end{lemma}
\begin{proof}
First, for each subgraph $H$ of $G$, each square-free $H$-word is also a square-free $G$-word.
Secondly,  if the vertices $i$ and $j$ of $G$ belong to
different connected components, then the letters $i,j \in A_n$ (where $n$ is the order of $G$)
never occur in the same $G$-word.
\end{proof}

\subsection{Negative cases}

Recall that $P_3$ has only three vertices and two edges $01$ and $12$.
Hence every second letter in a square-free $P_3$-word $w$ must be $1$,
and hence the word $w'$ obtained by deleting all occurrences of $1$ must
be a binary square-free word, and thus of length at most three.

The case for the graph $P_4$ is somewhat more complicated, but a systematic
study, by hand or aided by a computer, quickly ends.
The longest square-free $P_4$-words are of length 15. They are
$0121 0123 2101 210$ and the dual $321232101232123$ obtained by the permutation
$(0 \, 3)(1 \, 2)$.

\begin{lemma}\label{P4}
There are no square-free $P_4$-words of length 16.
\end{lemma}

\subsection{The cases $P_m$ for $m \geq 5$}

Consider the  morphism $\alpha\colon A_3^* \to A_3^*$ defined by
\begin{align*}
\alpha(0) & =201021202101201021012021 && \text{ of length  24},\\
\alpha(1) &= 2010212021012021  && \text{ of length  16},\\
\alpha(2) &= 20102101 && \text{ of length  8}.
\end{align*}

Note that the morphism $\alpha$ is not square free, since
\[\alpha(010)=2010212021012010210120(2120102120210120)^21021012021\,.\]
However, as we have seen, the Thue word $\Thue$ does not contain $010$.

The first case of the next lemma is seen to hold by checking
the short words by a computer program. The second claim is obvious, since
for $i \ne j$, the word $\alpha(i)$ does not have a suffix that is a prefix of $\alpha(j)$,
i.e., different words $\alpha(i)$ and $\alpha(j)$ do not overlap.
(Although $\alpha(2)$ is a factor of $\alpha(0)$.)

\begin{lemma}\label{le:1}
(a) \
If $w$ is of length 5 is square-free and does not contain $010$,
then also $\alpha(w)$ is square-free.

\smallskip
\noindent
(b)
The word $\alpha(i)$ with $i \in \{0,1\}$ is \emph{aligned in} $\alpha(A_3)^*$, i.e.,
if $w = i_1i_2 \cdots i_n \in A_3^*$ and $\alpha(w)=u_1\alpha(i)u_2$
then $u_1=\alpha(i_1 \cdots i_{k-1})$, $i=i_k$ and $u_2 =\alpha(i_{k+1} \cdots i_n)$.
\end{lemma}

\begin{lemma}\label{P5}
Let $w$ be a square-free word that does not have a factor $010$.
Then also $\alpha(w)$ is a square free.
For the Thue word $\Thue$,  the infinite word $\alpha(\Thue)$ is
a $3$-coloured square-free $P_5$-word.
\end{lemma}

\begin{proof}
Assume  that $\alpha(w)$ contains a nonempty square, say, for some $u \in A_3^*$
and letters $p$, $p_1$, $p_2 \in A_3$,
\begin{align*}
&\alpha(w)=w_1u^2w_2, \ \text{ where }
u = v_1\alpha(v)v_2 = v'_1\alpha(v')v'_2\,,\\
&v_2v'_1 = \alpha(p), \
v_1 \ \text{ is a suffix of }\alpha(p_1), \
v'_2 \ \text{ is a prefix of }\alpha(p_2)\,.
\end{align*}
By Lemma~\ref{le:1}(a), $v$ and $v'$ are nonempty, and by Lemma~\ref{le:1}(b),
$v=v'$  and also $v_1=v'_1$ and $v_2=v'_2$.
Now $\alpha(p)=v_2v_1$, where $v_1$ is a suffix and $v_2$ a suffix of some words
$\alpha(i)$ and $\alpha(j)$, respectively.
One sees that this is not the case for any $p,i,j$, simply because $\alpha(p)$ begins with
the special word $2010$ that occurs only as prefix of the words and
once in the middle of $\alpha(0)$.

Therefore,   $\alpha(\Thue)$ is square-free,
since all its finite prefixes are square-free
and $\Thue$ does not have any occurrences of the factor $010$.

For the second claim, consider the path $P_5$ with the edges $01,12,23,34$,
and let the colouring
of the vertices be $\varphi\colon A_5 \to A_3$ defined by
$\varphi(0)=1$,  $\varphi(1)=0$,  $\varphi(2)=2$,  $\varphi(3)=0$, $\varphi(4)=1$.
Then the morphism $\beta\colon A_3^* \to A_5^*$ defined by
\begin{align*}
\beta(0) & =201023202343201023432023\\
\beta(1) &=2010232023432023\\
\beta(2) &= 20102343\,
\end{align*}
satisfies $\alpha=\varphi\beta$. Since $\alpha(\Thue)$ is square-free, so is $\beta(\Thue)$.
Clearly $\beta(\Thue)$ is a $P_5$-word.
\end{proof}

\begin{corollary}\label{Pm}
We have $\gamma(G)=3$ for all connected graphs containing a subgraph $P_5$.
In particular, for each $P_n$  and $C_n$ with $n \geq 5$ there exists
an infinite 3-coloured square-free  $P_n$-word.
\end{corollary}

\begin{proof}
Indeed, each $P_n$ and $C_n$, with $n \geq 5$, has a subgraph equal to~$P_5$.
Moreover, we always have that if $G$ contains a $P_5$, then $\gamma(G) \geq 3$,
since there are no infinite square-free binary words. Hence the first claim holds.
\end{proof}

\subsection{The case for graphs with subgraphs $C_3, C_4$ or $K_{1,3}$}

We first consider the claw $K_{1,3}$.

\begin{lemma}\label{degree}
Let $G$ be a graph with a vertex of degree at least three.
Then there exists an infinite $4$-coloured square-free $G$-word.
Moreover, there does not exist any infinite $3$-coloured square-free $K_{1,3}$-words.
\end{lemma}
\begin{proof}
Without loss of generality, we can assume that the vertex~$3$ of~$G$
has the neighbours $0,1,2$.
Let $w$ be an infinite square-free word $w$ over $A_3$,
and consider the infinite word $w'$ obtained by replacing
each $i \in A_3$ by the word $i3$ of length two. It is clear that~$w'$ is still square-free
and also that it is a $G$-word, where four colours are present.

For the negative case, let $w= \varphi(u)$ for a $3$-colouring  $\varphi$ of $G$
and a $G$-word $u=3i_13i_2 \cdots$. Now  also $\varphi(i_1i_2 \cdots)$
must be square-free. However, we have two equicoloured neighbours $i$ and $j$
of the vertex~$3$, say $\varphi(i) = \varphi(j)$, and thus
$\varphi(i_1i_2 \cdots)$ is a binary word, and not square-free at all.
\end{proof}

By Lemma~\ref{degree}, we can assume that the maximum degree of the
graph $G$ is two, and hence the connected components are either paths or cycles.
 By Lemma~\ref{P4} and Corollary~\ref{Pm},
this leaves the connected graphs that are cycles $C_3$ or $C_4$.

\begin{lemma}\label{le:cycles}
There exists an infinite square-free $C_n$-graph in all cases $n \geq 3$.
We have $\gamma(C_3)=3$ and $\gamma(C_4)=4$.
\end{lemma}
\begin{proof}

For $n=3$, we observe that every square-free word over the alphabet $A_3$ is
also a $C_3$-word. Thus the case for $C_3$ is clear.

Let then $n=4$.
If $w \in A_{n-1}^*$ is a square-free $C_{n-1}$-word, then $w' \in A_n^*$
is a square-free $C_n$-word, where $w'$ is obtained by inserting
the letter $n-1$ between every pair $(0,n-2)$ and $(n-2,0)$.
Indeed, if $w'$ has a square $u^2$, then  the deletion of $n-1$
would yield a nonempty square in $w$; a contradiction.
This shows that $\gamma(C_4) \leq 4$.

Consider then a 3-colouring $\varphi$ of $C_4$, and assume $\varphi(w)$
is a square-free $C_4$-word.
Then two diagonal elements of the square $C_4$
must obtain the same colour, say $\varphi(1)=\varphi(3)$;
otherwise the  $C_4$-word coloured by $\varphi$ would be a $P_4$-word.
Now, the colour $\varphi(1)=\varphi(3)$ occurs in every second place in
$\varphi(w)$, and hence as in the proof of Lemma~\ref{degree}, the
case reduces to binary words, giving a  contradiction.
\end{proof}

We have now a simplified proof of a result due to Dean~\cite{Dean}.

\begin{corollary}[Dean]\label{Dean}
There exists an infinite square-free word $w$
that is reduced in the free group of two generators.
\end{corollary}

\begin{proof}
We apply Lemma~\ref{le:cycles} to $C_4$. Let $w$ be a square-free $C_4$-word.
One interprets~$2$ as the inverse element of the generator~$0$
and $3$ as the inverse element of the generator~$1$.
Since $0$ and $2$, and $1$ and $3$, respectively, are never adjacent in $w$
the word $w$ is reduced in the free group.
\end{proof}

We also can show the existence of an infinite square-free $C_4$-word
by considering a suitable \emph{uniform morphism},
where there is a constant $m$ such that $|\alpha(a)|=m$ for all letters $a$.
Define $\alpha\colon A_4^* \to A_4^*$ by
\begin{align*}
\alpha(0) &=010301210323\,,\\
\alpha(1) &=010301230323\,,\\
\alpha(2) &=010301232123\,,\\
\alpha(3) &=010321030123\,,
\end{align*}
where the images have length 12.
For this we rely on the following  theorem by Crochemore~\cite{Crochemore}.

\begin{theorem}[Crochemore]\label{Crochemore}
A uniform morphism  $h\colon A^* \to A^*$ is square-free
if and only if~$h$ preserves square-freeness of words of length~3.
\end{theorem}

One can deduce, or check by a computer,
that $\alpha$ preserves square-freeness of length three words. Therefore,
by Theorem~\ref{Crochemore}, $\alpha(w)$ is square-free for all square-free
infinite words $w$ over $A_4$. Clearly, $\alpha(w)$ is $C_4$-word.

\section{Tournament words}

The above problem can be modified for oriented graphs, i.e., directed graphs
$G=(A_n,E)$, where $ij \in E$ implies $ji \notin E$. A \emph{tournament}
is an orientation of a complete graph. A word $w \in A_n^*$
is a \emph{tournament word}, if for each different $i,j \in A_n$, if the word  $ij$
is a factor of $w$, then $ji$ is not a factor.

The case first case of the following result is obtained by a systematic computer search.

\begin{theorem}
(a)
The longest square-free tournament word over the four letter alphabet $A_4$ has length 20.
These longest words are $w=01201320120320132032$ and those obtained from $w$
by permuting the letters.

\smallskip
\noindent (b)
There exists infinite square-free  tournament words over $A_5$.
\end{theorem}

\begin{proof}
For the second claim, consider the uniform morphism $\alpha\colon A_3^* \to A_4^*$ defined by
\begin{align*}
\alpha(0) &= 0123014\,,\\
\alpha(1) &= 0130124\,,\\
\alpha(2) &= 0120134\,.
\end{align*}
The images $\alpha(i)$ have length $7$.
The morphism $\alpha$ can be easily seen to be square-free,
since the letter $4$ occurs only at the end of the images,
and in $\alpha(i)$ the letter is preceded by the letter $i$.
Hence if $w$ is an infinite square-free word over $A_3$, also
$\alpha(w)$ is square-free, and clearly it is a tournament word a $G$-word over $A_5$.
\end{proof}

\bibliographystyle{plain}
\bibliography{bib-sqfree}

\end{document}